\documentclass{amsart}
\usepackage{amsmath}
\usepackage{amssymb}
\usepackage{amsthm}
\begin{document}

\newcommand{\Isom}{\mathrm{Isom}}
\newcommand{\tr}{\mathrm{Trace}}
\newcommand{\divergence}{\mathrm{div} }
\newcommand{\id}{\mathrm{id}}
\newcommand{\Ric}{\mathrm{Ric}}

\theoremstyle{plain}
\newtheorem{theorem}{Theorem}[section]
\newtheorem{corollary}{Corollary}[section]
\newtheorem{proposition}{Proposition}[section]
\newtheorem{lemma}{Lemma}[section]
\newtheorem{conjecture}{Conjecture}[section]

\author{Jun-ichi Mukuno}

\address{Graduate School of Mathematics, Nagoya University, Chikusaku, Nagoya 464-8602, Japan}
 \thanks{This work is partially  supported by the Grant-in-Aid for JSPS Fellows No.\ 24-205. }
\email{m08043e@math.nagoya-u.ac.jp}
\title[Fundamental group of a
globally hyperbolic Lorentzian manifold]{On the  fundamental group of a  complete
globally hyperbolic Lorentzian manifold with a lower bound
for the curvature tensor }

\maketitle
\begin{abstract}
   In this paper, we study the fundamental group of
a certain class of   globally hyperbolic Lorentzian manifolds  with  a positive  curvature tensor. 
We prove that the fundamental group of  lightlike geodesically complete   
parametrized Lorentzian  products is finite under the conditions of a positive curvature tensor and the fiber compact.  
 \end{abstract}

\section{Introduction}
Calabi and Markus~{\cite{Ca}} stated the following theorem on the fundamental group of Lorentzian space forms with positive curvature:
\begin{theorem}[Calabi--Markus~{\cite[Theorem 1]{Ca}}]\label{tm:Ca-Ma}
Let $m$ be an integer greater than $2$.
Any $m$-dimensional  Lorentzian space form with positive curvature has a finite fundamental group.
\end{theorem}
Wolf~\cite{Wo} and Kulkarni~\cite{Kul_2} generalized Theorem~\ref{tm:Ca-Ma} for pseudo-Riemannian space forms,
and Kobayashi~\cite{Kob_2,MR1231232} studied the extension of Theorem~\ref{tm:Ca-Ma} to reductive and solvable homogeneous spaces.

Note that
Lorentzian space forms that have positive curvature are non-compact,
whereas Riemannian space forms with positive curvature are compact.
Theorem~\ref{tm:Ca-Ma}, however,  indicates that
there may be an analogy between the fundamental groups of Riemannian and Lorentzian manifolds satisfying
similar geometric assumptions.
In fact, Kobayashi asks whether the finiteness of the fundamental group still holds
if we perturb the metric of positive constant curvature.
In~\cite{Kob_7}, Kobayashi proposed the following conjecture in specific terms:
\begin{conjecture}[{Kobayashi~\cite[Conjecture 3.8.2]{Kob_7}}]\label{con:Kob}
Let $m$ and $q$ be positive integers with $m \geqslant 2q$.
Assume that $M$ is an $m$-dimensional  geodesically complete pseudo-Riemannian manifold of index $q$,
and that we have  a positive lower  bound on the sectional curvature of $M$.
Then,
\begin{enumerate}
\item[$\mathrm{(1)}$] $M$ is never compact;
\item[$\mathrm{(2)}$] if  $m \geqslant 3$, the fundamental group of $M$ is always finite.
\end{enumerate}
\end{conjecture}

Conjecture~\ref{con:Kob}  is analogous to  Myers' theorem in Riemannian geometry, and can therefore
be posed in terms of understanding the topology of pseudo-Riemannian manifolds of variable
curvature.
Against this aim, however,
Conjecture~\ref{con:Kob}   can be affirmatively solved.
A proof for the case $\dim (M) =2$ was given by Kulkarni~\cite[Corollary 2.10]{Kul_2}.
In the case $\dim (M) \geqslant 3 $,
 Kulkarni~\cite{MR522040}  proved that
the one-sided bound  of the sectional curvature
implies that the sectional curvature is constant.
The positive constant curvature case was proved by
Calabi--Markus~\cite{Ca} and Wolf~\cite{Wo},
and the proof of Conjecture~\ref{con:Kob}  is complete.

Unexpectedly, we have solved Conjecture~\ref{con:Kob} as a result of its curvature condition.
In this paper, we  reformulate Conjecture~\ref{con:Kob} for the Lorentzian case
to  cover some Lorentzian manifolds of variable curvature.
Instead of the one-sided bound on the sectional curvature,
Andersson and Howard~\cite{MR1664893} proposed
 the following curvature condition: there exists a  constant $k$ such that
 \begin{equation*}
 \langle R(u, v) v, u \rangle
\geqslant
k (\langle u, u \rangle \langle v, v \rangle - \langle u, v \rangle^{2})
\end{equation*}
for any tangent vectors $u, v$, where $R$ is the curvature tensor.
Following Andersson--Howard~\cite{MR1664893}, we denote this condition by $R \geqslant k$.
A subset $S$ in a time-oriented Lorentzian manifold $M$ is a \emph{Cauchy hypersurface} if
any inextensible timelike curve meets $S$ at a single point.
$M$ is  said to be \emph{globally hyperbolic}
if there exists a Cauchy hypersurface $S$ in $M$.
Bernal and S{\'a}nchez~\cite[Theorem 1.2]{MR2254187} proved that a globally hyperbolic Lorentzian manifold is diffeomorphic
to a product manifold of a timelike real line and some spacelike submanifold.
$M$ is  said to be \emph{non-spacelike} (resp.\  \emph{lightlike} )  \emph{geodesically complete}
if  any inextensible, non-spacelike (resp.\  lightlike ) geodesic is defined on the real line.
We pose the following:
\begin{conjecture}\label{conj:new}
Assume that $M$ is a non-spacelike geodesically complete globally hyperbolic Lorentzian manifold
with a compact Cauchy hypersurface
of    dimension $m \geqslant 3$.
Suppose that
 there exists a positive constant $k$ such that
$M$ satisfies the curvature condition $R \geqslant k$.
Then,  the fundamental group of $M$ is finite.
\end{conjecture}

 We now make some remarks on Conjecture~\ref{conj:new}.
First, the global hyperbolicity is necessary.
In fact,  $(S^{1} \times S^{m}, -g_{S^{1}}+ g_{S^{m}})$ is  a counterexample of Conjecture~\ref{conj:new}  when
 global hyperbolicity is excluded,     where
 $(S^{m}, g_{S^{m}})$ is an $m$-dimensional sphere with the standard metric.
Second, Beem and Ehrlich~\cite[Corollary 3.8]{MR898152}  proved that
the non-spacelike geodesic completeness is $C^1$-stable for globally hyperbolic manifolds.
Therefore, even if we perturb the metric,
 the assumptions of  Conjecture~\ref{conj:new}  hold.
Third, compared with Conjecture~\ref{con:Kob}, we have omitted spacelike
completeness.
For the global assumption of spacelike direction,
we require the compactness of Cauchy hypersurfaces.
Finally, any  Lorentzian space form of positive curvature satisfies the assumptions of Conjecture~\ref{conj:new}.
This follows from the proof of  Theorem~\ref{tm:Ca-Ma}.

 We now give a partial solution to Conjecture~\ref{conj:new}.
Let $F$ be a manifold, and $\{g_{t}\}_{t \in \mathbb{R}}$ be a smooth family of
Riemannian  metrics on $F$.
We call the Lorentzian manifold $(\mathbb{R} \times F,  -dt^{2} + g_{t})$
\emph{a  parametrized Lorentzian  product}, where
$t$ is the parameter of $\mathbb{R}$.
We call $F$ the \emph{fiber}.
Note that  a  parametrized Lorentzian  product with the fiber compact is globally hyperbolic.
We obtain the following theorem:
\begin{theorem}\label{tm:main}
Let $(M, g)$ be  a lightlike geodesically complete
parametrized Lorentzian product of dimension $m \geqslant 3$ with the fiber compact.
Suppose that there exists a positive constant $k$ such that $R \geqslant k$,
where $R$ is the curvature tensor of $M$.
Then,
 the fundamental group $\pi_{1}(M)$ is finite.
\end{theorem}
 A  parametrized Lorentzian  product includes a kind of timelike geodesic completeness,
since $\mathbb{R} \times \{p \}$ is a complete timelike geodesic for any $p \in F$.
This is considered to be   the reason   our theorem assumes only lightlike completeness.

\section{A necessary condition for the finiteness of the fundamental group}
In this section, we consider a general setting for Theorem~\ref{tm:main}.
We prove the following proposition:
\begin{proposition}\label{prop}
Let $M$ be a lightlike geodesically  complete  globally hyperbolic Lorentzian manifold.
Assume that $M$ satisfies $R \geqslant k$ for some positive constant $k$,
and suppose that there exists a closed spacelike Cauchy hypersurface $F$ in $M$ such
that the operator norm of the shape operator  is bounded above by $\sqrt{k}$.
Then, the fundamental group of $M$ is finite.
\end{proposition}

Let us  begin the proof of Proposition~\ref{prop}.
First, we calculate  the sectional curvature of $F$.
Let  $\nabla^{\top}$  be  the induced connection of  the hypersurface $F$,
and  $R^{\top}$  be the   curvature tensor  of $F$.
Recall the following, known as the \emph{Gauss formula}:
\begin{equation*}
g(R(u, v)v, u)
= g(R^{\top}(u, v)v, u)
+ g(S(u), u) g(S(v), v) -
g(S(u), v)^{2},
\end{equation*}
for any $u, v \in TF$.
Take any $2$-dimensional  subspace $\Pi$ in $TF$.
Let $(u, v)$ be an orthonormal basis of $\Pi$.
We have
\begin{align}
k (g(u, u)g(v, v)- g(u,v)^{2}) &\leqslant
g(R^{\top}(u, v)v, u)
+ g(S(u), u) g(S(v), v)
-g(S(u), v)^{2}  \nonumber \\
&\leqslant g(R^{\top}(u, v)v, u)
+ g(S(u), u) g(S(v), v), \label{ineq}
\end{align}
where $S$ is the shape operator of $F$.
The left-hand side of the inequalities is $k$.
The requirement of the shape operator  implies that  $|g(S(u), u) g(S(v), v)| \leqslant k$.
Therefore, the sectional curvature of $F$ is non-negative.

Next, we investigate the topology of $F$.
First, recall the following structure theorem for the fundamental group of  a closed  Riemannian manifold of non-negative curvature:
\begin{theorem}[Toponogov~\cite{MR0108808}, {Cheeger--Gromoll~\cite[Theorem 3]{MR0303460}}]\label{tm:split}
Let $M$ be a closed Riemannian manifold of non-negative sectional curvature.
Then, the universal covering Riemannian manifold $\widetilde{M}$ of $M$  can be split isometrically as $\mathbb{R}^{p} \times \widetilde{N}$,
where $\widetilde{N}$ is a closed Riemannian manifold.
Moreover, the fundamental group $\pi_{1}(M)$ includes a free abelian  subgroup $\mathbb{Z}^{p}$ of
finite index that acts properly discontinuously and cocompactly   
 as a  deck transformation on the Euclidean factor.
\end{theorem}

Let $g_{F}$ be the induced metric of $F$.
We know that $(F,g_{F})$ is a closed Riemannian manifold of non-negative curvature.
From Theorem~\ref{tm:split}, it follows that the universal covering Riemannian manifold $(\widetilde{F}, g_{\widetilde{F}})$ of
$(F, g_F)$ is the Riemannian product
manifold of the Euclidean space and some closed Riemannian manifold $\widetilde{N}$.
Therefore, we show that the dimension of the Euclidean factor is zero.
Suppose, by way of contradiction, that this dimension is not zero.
Then, the fundamental group $\pi_{1}(F)$ has a free abelian normal subgroup $\mathbb{Z}^{p}$ of finite index for some
$p > 0$.
Let $ (\overline{F}, g_{\overline{F}})$ be the quotient Riemannian  manifold $\widetilde{F}/\mathbb{Z}^{p-1}= \mathbb{R} \times
 \overline{N}$, where  $\overline{N}$ denotes $\mathbb{R}^{p-1}/\mathbb{Z}^{p-1} \times \widetilde{N}$.
Then, we have
the Riemannian covering map $\pi: (\overline{F}, g_{\overline{F}}) \rightarrow (F, g_{F})$.
Note that
$g_{\overline{F}}$ is represented as the Riemannian metric $ds^2 + g_{\overline{N}}$,
where $s$ is the parameter of $\mathbb{R}$ and $g_{\overline{N}}$ is the Riemannian metric of $\overline{N}$.

We now state a theorem of Bernal--S{\'a}nchez~\cite{MR2254187}:
\begin{theorem}[{Bernal--S{\'a}nchez~\cite[Theorem 1.2]{MR2254187}}]
Let $(M, g)$ be a globally hyperbolic Lorentzian manifold with a spacelike Cauchy
hypersurface $F$.
Then, there exists a smooth  function $\tau :M\rightarrow \mathbb{R}$ satisfying the following conditions:
\begin{itemize}
\item $\tau$ is a time function, i.e., $\tau$ is strictly increasing along any future directed timelike curve;
\item each level hypersurface  $\tau^{-1}(t)$ is a spacelike Cauchy hypersurface   for any $t \in \mathbb{R}$;
\item  $\tau^{-1}(0)=F$.
\end{itemize}
\end{theorem}
Let $\phi$ be the gradient flow $\phi : \mathbb{R} \times F \rightarrow M$ of   the time function $\tau$ in the above theorem,
and note that  $\phi$ is a diffeomorphism such that $\phi(\{0\} \times F)=F$.
We use the same letter $g$ for the Lorentzian metric of $\mathbb{R}\times F$ induced from $M$.
Then, the restricted metric $g|_{\{0\}\times F}$ is $g_{F}$.
Note that the covering map $\pi:\overline{F} \rightarrow F$ extends naturally to
the covering map $\id \times \pi:
\mathbb{R} \times \overline{F} \rightarrow \mathbb{R} \times F$.
We denote the Lorentzian manifold $(\mathbb{R} \times \overline{F}, (\id\times \pi)^{*}g)$ as $\overline{M}=(\mathbb{R} \times \overline{F}, g_{\overline{M}})$.

Let $P$ be a submanifold of codimension $2$ in $M$.
Then, there exist two linearly independent
future directed lightlike vector fields $l^{+}, l^{-}$
perpendicular to $P$.
$P$ is called a \emph{trapped surface} if
$\divergence_{P} (l^+)>0$ and $\divergence_{P} (l^-)>0$.
Let us recall the Penrose singularity theorem:

\begin{theorem}[Penrose~\cite{MR0172678}]
Let $M$ be a globally hyperbolic Lorentzian manifold  of dimension $m \geqslant 3$
with a  non-compact Cauchy hypersurface.
Assume that  $M$ satisfies the following two conditions:
\begin{itemize}
\item $\Ric (u, u) \geqslant 0$ for any future directed lightlike tangent vector $u$;
\item $P$ is a compact trapped surface.
\end{itemize}
Then,  $M$ is not lightlike geodesically complete.
\end{theorem}

Note that the curvature condition of Proposition~\ref{prop} implies  $\Ric (u, u) \geqslant 0$ for any lightlike tangent vector $u$.
Let  $N_{0}$ be the submanifold $\{0\} \times \{0\} \times \overline{N}$ of $\overline{M}=(\mathbb{R} \times \mathbb{R} \times \overline{N}, g_{\overline{M}})$.
We write  $l_{+}$ and  $l_{-}$   for
the normal lightlike tangent factors $n + \partial / \partial s$ and
 $n- \partial / \partial s$, respectively,
 where $n$ is the unit normal vector of the Cauchy hypersurface $\{ 0\} \times \mathbb{R} \times \overline{N}$,
 and $s$ is the parameter of the real line $\mathbb{R}$ of  $\{ 0\} \times\mathbb{R} \times \overline{N}$.
Take an orthonormal basis $\{e_{i} \}_{i=1}^{m-1}$ of the tangent space of $N_{0}$.
Since
$N_{0}$ is a totally geodesic submanifold in
 $\{ 0\} \times \mathbb{R} \times \overline{N}$,
we have
\begin{align*}
\sum_{i=1}^{m-1} g(\nabla_{e_{i}} l_{+}, e_{i})
&=
\sum_{i=1}^{m-1} g(\nabla_{e_{i}} l_{-}, e_{i})
=
\sum_{i=1}^{m-1} g(\nabla_{e_{i}}  n , e_{i})\\
&=
\sum_{i=1}^{m-1} g(S(e_{i}), e_{i}).
\end{align*}

 We should remark that $g(R^{\top}(u, \partial/ \partial s) \partial / \partial s, u)=0$ for
 any unit tangent vector $u$ of $N_{0}$.
Therefore, using inequality (\ref{ineq}), we have
$g(S(\partial/ \partial s),  \partial/ \partial s) g(S(u),  u) =k$.
Without loss of generality, we can assume that $ g(S(\partial/ \partial s),  \partial/ \partial s)=\sqrt{k}$.
Then, we obtain
\begin{equation*}
\sum_{i=1}^{m-1} g(S(e_{i}), e_{i})=(m-1) \sqrt{k}>0.
\end{equation*}
Thus, $N_{0}$ is a trapped submanifold.
From  the  Penrose singularity theorem,    $\overline{M}$ is not lightlike geodesically complete.
This is a contradiction, and Proposition~\ref{prop} has been proved.

\section{Proof of Theorem~\ref{tm:main}}
In this section, we prove Theorem~\ref{tm:main}.
During the proof,
$M$ is a parametrized Lorentzian manifold $(\mathbb{R} \times F, g=-dt^2 + g_t)$.
For any $p\in F$,  we define the curve $\gamma_{p}: \mathbb{R} \rightarrow M$ by
$\gamma_{p}(t)=(t,p)$. It is easy to check that $\gamma_{p}$ is a timelike geodesic.
We denote   the hypersurface $\{t\} \times F$ of $M$ as $F_{t}$;
note that $\partial / \partial t$ is a normal vector of $F_{t}$, denoted by $n$.
We define a second fundamental form   $S : TF_{t} \rightarrow TF_{t}$
by $S(u)=\nabla_{u} n$, and the curvature operator $R_{n} : TF_{t} \rightarrow TF_{t}$  by
$R_{n}(u)=R(u, n) n$.
Then, we have the following Riccati equation:
\begin{equation}
\nabla_{n}(S) + S^{2} + R_{n} = 0.
\end{equation}
We obtain the following lemma:
\begin{lemma}\label{le:boundedness}
For any $p \in F$ and
any $u \in T_{p}F$ with $g_{0}(u,u)=1$,   let $\widetilde{u}(t)$ be the parallel vector field along the timelike geodesic $\gamma_{p}(t)$ in $M$
such that    $\widetilde{u}(0)=(0, u) \in
T_{0}\mathbb{R} \oplus T_{p}F =
T_{\gamma_{p}(0)}\mathbb{R} \times F$.
Then, we have
\begin{equation*}
|g(S(\widetilde{u}(t)),  \widetilde{u}(t) )| \leqslant \sqrt{k},
\end{equation*}
for any $t \in \mathbb{R}$.
\end{lemma}
\begin{proof}\renewcommand{\qedsymbol}{}
The Riccati equation implies
\begin{equation*}
\frac{\partial}{\partial t}   g(S(\widetilde{u}(t)),     \widetilde{u}(t))
= - g(R_{n}(\widetilde{u}(t)),     \widetilde{u}(t)) -g(S(\widetilde{u}(t)),    S(\widetilde{u}(t)) ).
\end{equation*}
Using the Cauchy--Schwartz inequality, we have
$g(S(\widetilde{u}(t)),    S(\widetilde{u}(t)) ) \geqslant |g(S(\widetilde{u}(t)),    \widetilde{u}(t)) |^{2}$.
From the curvature condition, it follows that    $g(R_{n}(\widetilde{u}(t)),     \widetilde{u}(t)) \geqslant - k$.
Setting $f(t)=   g(S(\widetilde{u}(t)),     \widetilde{u}(t)) $,
 we obtain the following inequality:
\begin{equation*}
\frac{\partial}{ \partial t} f(t) \leqslant k - f(t)^{2}
\end{equation*}
for any $t \in \mathbb{R}$.

Let us now suppose that there exists $t_{0}$ such that
$|f(t_{0})|=| g(S(\widetilde{u}(t_{0})),     \widetilde{u}(t_{0}))  | >\sqrt{k}$.
Without loss of generality, we can assume that $f(t_{0}) > \sqrt{k}$.
Then,  there exists a positive number $t_{1}$ such that
$f(t_{0})=\sqrt{k} \coth (\sqrt{k} t_{1})$.
By the  Riccati argument, we have
\begin{equation*}
f(t_{0}+ t) \geqslant  \sqrt{k} \coth (\sqrt{k} (t_{1}+t))
\end{equation*}
for any $t \leqslant 0$.
However the right-hand side of the inequality  goes to  infinity
as $t$ approaches $-t_{1}$.
This is a contradiction, and the proof of Lemma~\ref{le:boundedness} is complete.
\end{proof}
Theorem~\ref{tm:main} follows from Proposition~\ref{prop} and Lemma~\ref{le:boundedness}.

\providecommand{\bysame}{\leavevmode\hbox to3em{\hrulefill}\thinspace}
\providecommand{\MR}{\relax\ifhmode\unskip\space\fi MR }
\providecommand{\MRhref}[2]{%
  \href{http://www.ams.org/mathscinet-getitem?mr=#1}{#2}
}
\providecommand{\href}[2]{#2}

\end{document}